\documentclass{amsart}

\usepackage{amssymb}
\usepackage{overpic}
\usepackage{graphics}
\usepackage{epsfig}
\usepackage{amsmath}
\usepackage{amsfonts}
\usepackage{paralist}

\newtheorem{theorem}{Theorem}[section]
\newtheorem{corollary}{Corollary}

\newtheorem{lemma}[theorem]{Lemma}
\newtheorem{proposition}{Proposition}

\theoremstyle{definition}
\newtheorem{definition}[theorem]{Definition}
\newtheorem{remark}{Remark}

\begin{document}

\title{Isospectral Graph Reductions}

\author{L. A. Bunimovich}
\address{ABC Math Program and School of Mathematics, Georgia Institute of Technology, 686 Cherry Street, Atlanta, GA 30332}
\email{bunimovich@math.gatech.edu}

\author{B. Z. Webb}
\address{ABC Math Program and School of Mathematics, Georgia Institute of Technology, 686 Cherry Street, Atlanta, GA 30332}
\email{bwebb@math.gatech.edu}


\maketitle

\begin{abstract}
Let G be an arbitrary finite weighted digraph with weights in the set of
complex rational functions. A general procedure is proposed which allows for the reduction of G to a smaller graph with a less complicated structure having the same  spectrum as of G (up to some set known in advance). The proposed procedure has a lot of flexibility and could be used e.g. for design of networks with prescribed spectral and dynamical properties. 
\end{abstract}

\section{Introduction} 
The structure of a given graph can range in complexity from being quite simple, having some regular features or small edge and vertex sets to being extremely complicated where basic characteristics of the graph are hard to obtain or estimate. Such complicated structure is typical if for instance the graph represents some real network \cite{Albert02,Dorogovtsev03,Faloutsos99,Newman06,Strogatz03,Watts99}. 

An important problem therefore is whether it is possible to simplify or reduce a graph while maintaining its basic graph structure as well as some characteristic(s) of the original graph. A related key question then is which characteristic(s) to conserve while reducing a graph.

Studies of dynamical networks (i.e. networks of interacting dynamical systems which could be cells, power stations, etc.) reveal that an important characteristic of a network's structure is the spectrum of the network's adjacency matrix \cite{Blank06,Newman06,Afriamovich07,Motter07}. With this in mind we present an approach which allows for the reduction of a general weighted digraph in such a way that the spectrum of the graph's (weighted) adjacency matrix is maintained up to some well defined and known set. 

We denote the class of graphs for which reductions are possible by $\mathbb{G}$ which consists of the set of all finite weighted digraphs without parallel edges but possibly with loops having weights in the set $\mathbb{W}$ of complex rational functions. A graph $G\in\mathbb{G}$ can therefore be written as the triple $G=(V,E,\omega)$ where $V$ and $E$ are the vertex and edge sets of $G$ respectively and $\omega:E\rightarrow\mathbb{W}$. Each such graph has an adjacency matrix with a well defined spectrum which we denote by $\sigma(G)\subset\mathbb{C}$. With this in place a graph reduction of $G$ can be described as follows. 

Given a specific subset $S\subseteq V$ which we call a \textit{structural sets} of $G$, an \textit{isospectral reduction} of $G$ over the vertex set $S$ is a weighted digraph $\mathcal{R}_S(G)=(S,\mathcal{E},\mu)$ (see section 3 for the exact definitions). The main result of the paper is the following theorem (see theorem \ref{theorem1}).\\

\noindent \textbf{Theorem:} Let $G\in\mathbb{G}$ and $S$ be a structural set of $G$. Then $\sigma(G)$ and $\sigma(\mathcal{R}_S(G))$ differ at most by $\mathcal{N}(G;S)$.\\ 

\noindent The set $\mathcal{N}(G;S)$ is a finite set of complex numbers which is known and is the largest set by which $\sigma(G)$ and $\sigma(\mathcal{R}_S(G))$ can differ.

As a typical graph has many different structural sets it is possible to consider different isospectral reductions of the same graph. Moreover, since a reduced graph is again a weighted digraph it is possible to consider sequences of such reductions. The flexibility of this process is reflected in the fact that for a typical graph $G\in\mathbb{G}$ it is possible to reduce $G$ to a graph on any nonempty subset of its original vertex set. That is, we may simplify the structure of $G$ to whatever degree we desire.

From this it follows that if $\mathcal{V}$ is any nonempty subset of the vertices of $G$ then there are typically multiple ways to sequentially reduce $G$ to a graph on $\mathcal{V}$. As it turns out each such reduction results in the same graph independent of the particular sequence. This uniqueness result can be interpreted as the property that sequential reductions on $\mathbb{G}$ are commutative.

Furthermore, the class of graphs which can be reduced via this method is very general. Specifically, we may reduce those graphs in $\mathbb{G}$ which consist of weighted digraphs without parallel edges. Since undirected and unweighted graphs or graphs with parallel edges can be considered as weighted digraphs via some standard conventions then such graphs are automatically included as special cases of graphs that may be reduced by this procedure.

Because of the flexibility in reducing a graph the relation of having the same branch reduction is not an equivalence relation on the set $\mathbb{G}$ as this relation is not transitive. However, it is possible to construct specific types of structural sets as well as rules for sequential reductions which do induce equivalence relations on the graphs in $\mathbb{G}$ which we give examples of.

Also we note that generally, the tradeoff in reducing a graph is that although the graph structure becomes simpler the weights of edges become more complicated. Therefore, we also consider graph reductions over fixed weight sets in which the weight set of the graph is maintained under this reduction while the vertex set is reduced. 

The structure of the paper is as follows. In Sect. 2 we present notation and some general definitions. Sect. 3 contains the description of the reduction procedure as well as some results on sequences of such reductions. Some of the main results and techniques of the paper are then stated in Sect. 4. Proofs of these statements are given in Sect.5. In Sect. 6 we study the relations between the strongly connected components of the graph and its reductions. Sect. 7 considers reductions over fixed weight sets and Sect. 8 contains some concluding remarks.

\section{Preliminaries}

In what follows we formally consider the class of digraphs consisting of all finite weighted digraphs with or without loops having edge weights in the set $\mathbb{W}$ of complex rational functions described below. We denote this class of graphs by $\mathbb{G}$.

As previously mentioned, graphs or which are either undirected, unweighted, or have parallel edges can be considered as graphs in $\mathbb{G}$. This is done by making an undirected graph $G$ into a directed graph by orienting each of its edges in both directions. Similarly, if $G$ is unweighted then it can be made weighted by giving each edge unit weight. Also multiple edges between two vertices of $G$ may be considered as a single edge by adding the weights of the multiple edges and setting this to be the weight of this single equivalent edge. We will typically assume that the graph $G\in\mathbb{G}$ or use these conventions to make it so. 

By way of notation we let the digraph $G$, possibly with loops, be the pair $(V,E)$ where $V$ and $E$ are the finite sets denoting the \textit{vertices} and \textit{edges} of $G$ respectively, the edges corresponding to ordered pairs $(v,w)$ for $v,w\in V$. Furthermore, if $G$ is a \textit{weighted digraph} with weights in $\mathbb{W}$ then $G=(V,E)$ together with a function $\omega:E\rightarrow\mathbb{W}$ where $\omega(e)$ is the \textit{weight} of the edge $e$ for $e\in E$. We use the convention that $\omega(e)=0$ if and only if $e\notin E$. Importantly, if $G\in\mathbb{G}$ then similar to digraphs we will denote this by writing $G=(V,E,\omega)$. 

In order to describe the set of weights $\mathbb{W}$ let $\mathbb{C}[\lambda]$ denote the set of polynomials in the single complex variable $\lambda$ with complex coefficients. We define the set $\mathbb{W}$ to be the set of rational functions of the form $p/q$ where $p,q\in\mathbb{C}[\lambda]$ such that $p$ and $q$ have no common factors and $q$ is nonzero.

The set $\mathbb{W}$ is then a field under addition and multiplication with the convention that common factors are removed when two elements are combined. That is, if $p/q,r/s\in\mathbb{W}$ then $p/q+r/s=(ps+rq)/(qs)$ where the common factors of $ps+rq$ and $qs$ are removed. Similarly, in the product $(pr)/(qs)$ of $p/q$ and $r/s$ the common factors of $pr$ and $qs$ are removed. However, we may at times leave sums and the products of sums of elements in $\mathbb{W}$ uncombined and therefore possibly unreduced but this is purely cosmetic since there is one reduced form for any element in $\mathbb{W}$.

To introduce the spectrum associated to a graph having weights in $\mathbb{W}$ we will use the following notation. If the vertex set of the graph $G=(V,E,\omega)$ is labeled $V=\{v_1,\dots,v_n\}$ then we denote the edge $(v_i,v_j)$ by $e_{ij}$. The matrix $M(G)=M(G,\lambda)$ defined entrywise by $$\big(M(G)\big)_{ij}=\omega(e_{ij})$$
is the \textit{weighted adjacency matrix} of $G$.

We let the \textit{spectrum} of a matrix $A=A(\lambda)$ with entries in $\mathbb{W}$ be the solutions including multiplicities of the equation 
\begin{equation}\label{eq1}
\det(A(\lambda)-\lambda I)=0
\end{equation}
and for the graph $G$ we let $\sigma(G)$ denote the spectrum of $M(G)$. The spectrum of a matrix with entries in $\mathbb{W}$ is therefore a generalization of the spectrum of a matrix with complex entries. 

Moreover, the spectrum is a \textit{list} of numbers. That is, 
$$\sigma(G)=\big\{\ (\sigma_i,n_i):1\leq i\leq p,\sigma_i\in\mathbb{C},n_i\in\mathbb{N}\big\}$$
where $n_i$ is the multiplicity of the solutions $\sigma_i$ to equation (\ref{eq1}), $p$ the number of distinct solutions, and $(\sigma_i,n_i)$ the elements in the list. In what follows we may write a list as a set with multiplicities if this is more convenient. 

\section{Graph Reductions}

In this section we describe the main results of the paper. That is, we present a method which allows for the reduction a graph while maintaining the graph's spectrum up to some known set. We also give specific examples of this process notably using this method to reduce graphs associated with the Laplacian matrix of a graph. Some natural consequences and extensions of this process are also mentioned. 

\subsection{Setup} 
Here we first introduce some definitions as well as some terminology that allow us to be precise in our formulation of an isospectral reduction. 

In the following if $S\subseteq V$ where $V$ is the vertex set of a digraph let $\bar{S}$ denote the complement of $S$ in $V$. Also, as is standard, a \textit{path} $P$ in a digraph $G=(V,E)$ is a sequence of distinct vertices $v_1,\dots,v_m\in V$ such that $(v_i,v_{i+1})\in E$ for $1\leq i\leq m-1$ and in the case that the vetices $v_1,\dots,v_m$ are distinct, except that $v_1=v_m$, $P$ is a \textit{cycle}. Moreover, let the vertices $v_2,\dots,v_{m-1}$ of $P$ be the \textit{interior} vertices of $P$.

\begin{definition}\label{def1}
For $G=(V,E)$ let $\ell(G)$ be the digraph $G$ with all loops removed. We say the nonempty vertex set $S\subseteq V$ is a \textit{structural set} of $G$ if $\bar{S}$ induces no cycles in $\ell(G)$ and for each $v_i\in\bar{S}$, $\omega(e_{ii}\neq \lambda)$. We denote by $st(G)$ the set of all structural sets of $G$.
\end{definition}

\begin{figure}
  \begin{center}
    \begin{overpic}[scale=.5]{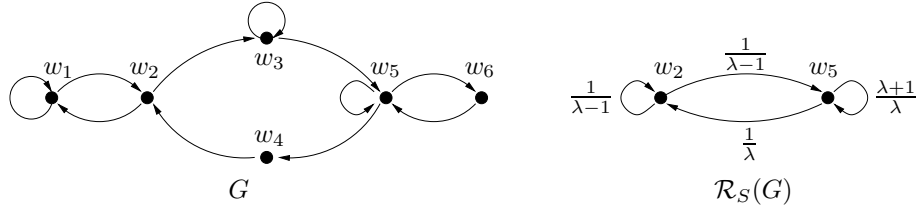}
    \put(25.5,-4){$G$}
    \put(4,10.5){$w_1$}
    \put(14,10.5){$w_2$}
    \put(75,10.5){$w_2$}
    \put(28.5,11.5){$w_3$}
    \put(28.5,2.5){$w_4$}
    \put(42,10.5){$w_5$}
    \put(93,10.5){$w_5$}
    \put(53,10.5){$w_6$}
    \put(83,12){{$\frac{1}{\lambda-1}$}}
    \put(85,1.5){$\frac{1}{\lambda}$}
    \put(65,7){$\frac{1}{\lambda-1}$}
    \put(100.5,7){$\frac{\lambda+1}{\lambda}$}
    \put(82,-4){$\mathcal{R}_S(G)$}
    \end{overpic}
  \end{center}
  \caption{Reduction of $G$ over $S=\{w_2,w_5\}$.}
\end{figure}

\begin{definition}
For $G=(V,E)$ with $S=\{v_1,\dots,v_m\}\in st(G)$ let $\mathcal{B}_{ij}(G;S)$ be the set of paths or cycles from $v_i$ to $v_j$ in $G$ having no interior vertices in $S$. Furthermore, let $$\mathcal{B}_S(G)=\bigcup_{1\leq i,j \leq m} \mathcal{B}_{ij}(G;S).$$ We call the set $\mathcal{B}_S(G)$ the set of all \textit{branches} of $G$ with respect to $S$. 
\end{definition}

\begin{definition}\label{branchprod}
Let $G=(V,E,\omega)$ and $\beta \in \mathcal{B}_S(G)$ for some $S\in st(G)$. If $\beta=v_1,\dots,v_m$ for $m>2$ we define $$\mathcal{P}_{\omega}(\beta)=\omega(e_{12})\prod_{i=2}^{m-1}\frac{\omega(e_{i,i+1})}{\lambda-\omega(e_{ii})}$$ as the \textit{branch product} of $\beta$. If $m=2$ we define $\mathcal{P}_{\omega}(\beta)=\omega(e_{12})$.
\end{definition}

\begin{definition}\label{reductiondef}
Let $G=(V,E,\omega)$ with structural set $S=\{v_1\,\dots,v_m\}$. Define $\mathcal{R}_S(G)=(S,\mathcal{E},\mu)$ to be the digraph such that $e_{ij}\in \mathcal{E}$ if $\mathcal{B}_{ij}(G;S)\neq\emptyset$ and 
\begin{equation}
\mu(e_{ij})=
  \sum_{\beta\in\mathcal{B}_{ij}(G;S)}\mathcal{P}_\omega(\beta), \ \ 1\leq i,j\leq m.
\end{equation}
We call $\mathcal{R}_S(G)$ the \textit{isospectral reduction} of $G$ over $S$.
\end{definition} The graph $\mathcal{R}_S(G)\in\mathbb{G}$ since $\mathcal{P}_\omega(\beta)\in\mathbb{W}$ implying $\mu(e_{ij})$ is as well. Figure 1 gives an example of a reduction of the graph $G$. We note here that all figures in this paper follow the aforementioned conventions that undirected and unweighted edges are assumed to be oriented in both directions and have unit weight.

As another example of a graph reduction consider the complete undirected unweighted graph without loops $K_n=(V,E)$ on $n$ vertices. If $V=\{v_1,\dots,v_n\}$ then $K_n$ has $n$ structural sets $S_k$ each given by $S_k=V\setminus\{v_k\}$, $1\leq k\leq n$. For each $k$ the graph $\mathcal{R}_{S_k}(K_n)$ has an $(n-1) \times (n-1)$ adjacency matrix $\mathcal{M}$ where $(\mathcal{M})_{ij}=1+1/\lambda$ for all $1\leq i,j\leq n-1$. 

For the complete bipartite graph $K_{m,n}=(V,E)$ where $V$ is partitioned into the sets $M$ and $N$ having $m$ and $n$ vertices respectively it follows that both $M,N\in st(K_{m,n})$. Moreover, $\mathcal{R}_M(K_{m,n})$ is the digraph with all possible edges including loops on $m$ vertices each having weight $n/\lambda$ (see figure 2). 

In order to understand the extent to which the spectrum of a graph is maintained under different reductions we introduce the following. If $S$ is a structural set of the graph $G=(V,E,\omega)$ where $V=\{v_1,\dots,v_n\}$ let $$\mathcal{N}(G;S)=\bigcup_{v_i\in \bar{S}}\{\lambda\in\mathbb{C}:\lambda=\omega(e_{ii}) \ \text{or} \ \omega(e_{ii}) \ \text{is undefined}\}.$$ 
That is, $\mathcal{N}(G;S)$ is the set of $\lambda\in\mathbb{C}$ for which there is some vertex $v_i$ off the structural set $S$ where $\omega(e_{ii})=\lambda$ or, as $\omega(e_{ii})=p_i(\lambda)/q_i(\lambda)\in\mathbb{W}$, the values of $\lambda$ at which $q_i(\lambda)=0$. As an example for the graph $G$ and structural set $S$ in figure 1 $\mathcal{N}(G;S)=\{0,1\}$. 

\begin{figure}
  \begin{center}
    \begin{overpic}[width=80mm, height=30mm]{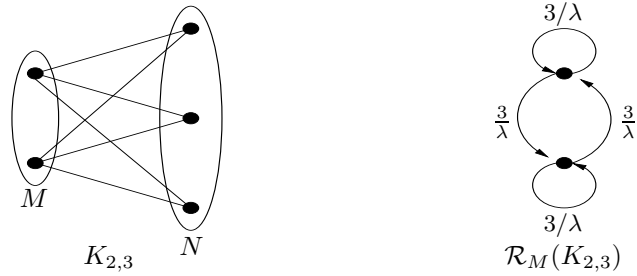}
    \put(12,-5.5){$K_{2,3}$}
    \put(1.5,4){$M$}
    \put(28,-4){$N$}
    \put(82,-5.5){$\mathcal{R}_M(K_{2,3})$}
    \put(80,17){$\frac{3}{\lambda}$}
    \put(101,17){$\frac{3}{\lambda}$}
    \put(88.5,0){\smaller$3/\lambda$}
    \put(88.5,35){\smaller$3/\lambda$}
    \end{overpic}
  \end{center}
  \caption{Reduction of $K_{2,3}$ over $M$.}
\end{figure}

If $G,H\in\mathbb{G}$ and $N\subseteq\mathbb{C}$ then let $\sigma(G)\setminus N$ be the list given by 
$$\sigma(G)\setminus N=\big\{ (\sigma_i,n_i)\in\sigma(G):\sigma_i\notin N \big\}.$$ 
Moreover, if it happens that $\sigma(G)\setminus N=\sigma(H)\setminus N$ then we say $\sigma(G)$ and $\sigma(H)$ differ at most by $N$. The main result of this paper can then be phrased as follows.

\begin{theorem}\label{theorem1}
Let $G\in\mathbb{G}$ with $S\in st(G)$. Then $\sigma(G)$ and $\sigma(\mathcal{R}_S(G))$ differ at most by $\mathcal{N}(G;S)$.
\end{theorem}

That is, the spectrum of $G$ and the spectrum of its reduction $\mathcal{R}_S(G)$ differ at most by elements of $\mathcal{N}(G;S)$ which justifies our use of the terminology isospectral reduction as the spectrum is preserved up to some known set.

We note that two weighted digraphs $G_1=(V_1,E_1,\omega_1)$, and $G_2=(V_2,E_2,\omega_2)$ are \textit{isomorphic} if there is a bijection $\rho:V_1\rightarrow V_2$ such that there is an edge $e_{ij}$ in $G_1$ from $v_i$ to $v_j$ if and only if there is an edge $\tilde{e}_{ij}$ between $\rho(v_i)$ and $\rho(v_j)$ in $G_2$ with $\omega_2(\tilde{e}_{ij})=\omega_1(e_{ij})$. If the map $\rho$ exists it is called an \textit{isomorphism} and we write $G_1\simeq G_2$.

\begin{definition} Let $G,H\in\mathbb{G}$. We say $G$ and $H$ have a reduction in common via the structural sets $S$ and $T$ respectively if $S\in st(G)$, $T\in st(H)$ and $\mathcal{R}_S(G)\simeq\mathcal{R}_T(H)$. 
\end{definition}

The fact that isomorphic graphs have the same spectrum with theorem \ref{theorem1} together imply the following. 

\begin{corollary}
If $G,H\in\mathbb{G}$ have a reduction in common via the structural sets $S$ and $T$ respectively then $\sigma(G)$ and $\sigma(H)$ differ at most by $\mathcal{N}(G;S)\cup\mathcal{N}(H;T)$.
\end{corollary}

Figure 3 gives an example of a reduction of $H$ in which $\mathcal{R}_T(H)\simeq\mathcal{R}_S(G)$ where $G$ and $S$ are the graph and structural set respectively in figure 1. Therefore, the adjacency matrices of $G$ and $H$ in figures 1 and 3 respectively have the same spectrum up to $\mathcal{N}(G;S)\cup\mathcal{N}(H;T)=\{0,1\}$. In this case one can compute $\sigma(G)=\{2,-1,1,1,0,0\}$, $\sigma(H)=\{2,-1,1,0\}$, and $\sigma(\mathcal{R}_S(G))=\sigma(\mathcal{R}_T(H))=\{2,-1\}$ 

\begin{figure}
  \begin{center}
    \begin{overpic}[scale=.5]{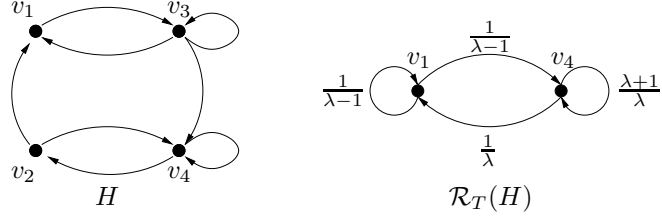}
    \put(14,-5){$H$}
    \put(0,-.5){$v_2$}
    \put(0,27){$v_1$}
    \put(26,27){$v_3$}
    \put(26,-0.5){$v_4$}
    \put(73,-5){$\mathcal{R}_T(H)$}
    \put(51.5,13){$\frac{1}{\lambda-1}$}
    \put(66,18.5){$v_1$}
    \put(76,22){$\frac{1}{\lambda-1}$}
    \put(78,3){$\frac{1}{\lambda}$}
    \put(90,18.5){$v_4$}
    \put(101,13){$\frac{\lambda+1}{\lambda}$}
    \end{overpic}
  \end{center}
  \caption{$\mathcal{R}_T(H)\simeq\mathcal{R}_S(G)$ from Fig. 1}
\end{figure}

We note that the representation of matrices with nonnegative entries by smaller matrices with polynomial entries has been used before (see e.g. \cite{Boyle02}). However, the reason for doing so is different from our motivation in this paper. 

\subsection{Laplacian Matrices}
An alternate view of the graph reductions presented in section 3.1 is to consider the reduction process one in which the matrix $M(G)$ is reduced to the matrix $M(\mathcal{R}_S(G))$ and view theorem \ref{theorem1} as a theorem about matrix reductions. Viewed this way, an important application of theorem \ref{theorem1} is that one may reduce not only the graph $G$ but also the graphs associated with both the combinatorial Laplacian matrix and the normalized Laplacian matrix of $G$. 

To make this precise, let $G=(V,E)$ be an unweighted undirected graph without loops, i.e. a \textit{simple graph}. If $G$ has vertex set $V=\{v_1,\dots,v_n\}$ and $d(v_i)$ is the degree of vertex $v_i$ then its \textit{combinatorial Laplacian matrix} $M_L(G)$ is given by 
$$M_L(G)_{ij}=\begin{cases}
d(v_i) & i=j\\
-1 & i\neq j \ \text{and} \ v_i \ \text{is adjacent to} \ v_j\\
0 & \text{otherwise}
\end{cases}$$ On the other hand the \textit{normalized Laplacian matrix} $M_\mathcal{L}(G)$ of $G$ is defined as
$$M_\mathcal{L}(G)_{ij}=\begin{cases}
1 & i=j \ \text{and} \ d(v_j)\neq 0\\
\frac{-1}{\sqrt{d(v_i)d(v_j)}} & v_i \ \text{is adjacent to} \ v_j\\
0 & \text{otherwise}
\end{cases}$$ 

The interest in the eigenvalues of $M_L(G)$ is that $\sigma(M_L(G))$ gives structural information about $G$ (see \cite{Chung97}). On the other hand knowing $\sigma(M_\mathcal{L}(G))$ is useful in determining the behavior of algorithms on the graph $G$ among other things (see \cite{Chung06}).

As every $n\times n$ matrix with weights in $\mathbb{W}$ has a unique weighted digraph associated to it then let $L(G)$ be the graph with adjacency matrix $M_L(G)$ and similarly let $\mathcal{L}(G)$ be the graph with adjacency matrix $M_\mathcal{L}(G)$. Since both $L(G)$ and $\mathcal{L}(G)$ can be considered in $\mathbb{G}$ via our conventions then either may be reduced. We summarize this as the following theorem which is a corollary to theorem \ref{theorem1}.

\begin{theorem}
Suppose $G$ is a simple graph with vertex set $V$. If $S\subseteq V$ is such that $\bar{S}$ induces no cycles in $G$ then $S\in st(L(G))$ and $\sigma(L(G))$ and $\sigma(\mathcal{R}_S(L(G)))$ differ at most by $\mathcal{N}(L(G);S)$. Similarly, $S\in st(\mathcal{L}(G))$ and $\sigma(\mathcal{L}(G))$ and $\sigma(\mathcal{R}_S(\mathcal{L}(G)))$ differ at most by $\mathcal{N}(\mathcal{L}(G);S)$. 
\end{theorem}

For example if $G=K_3$ is the complete graph on 3 vertices then the graph $L(G)$, shown in figure 4, has the structural set $S=\{v_1,v_2\}$ since $\bar{S}=\{v_3\}$ induces no cycles in $G$. Reducing over this set yields $\mathcal{R}_S\big(L(G)\big)$ where, as can be verified, $\mathcal{N}(L(G),S)=2\notin\sigma(L(G))$ implying $\sigma\big(\mathcal{R}_{S}(L(G))\big)=\sigma(L(G))$. 

\begin{figure}
  \begin{center}
    \begin{overpic}[scale=.5]{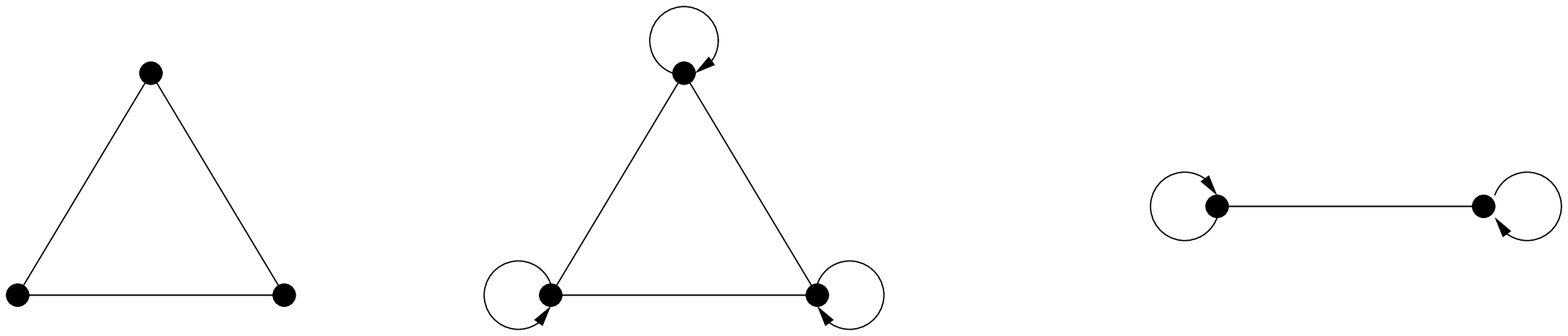}
    \put(8,-2.5){$G$}
    \put(40,-2.5){$L(G)$}
    \put(77,-2.5){$\mathcal{R}_{\{v_1,v_2\}}\big(L(G)\big)$}
    \put(-1,4){$v_1$}
    \put(17,4){$v_2$}
    \put(5,16){$v_3$}
    \put(33.5,5){$v_1$}
    \put(50.5,5){$v_2$}
    \put(38,16){$v_3$}
    \put(34,10){$-1$}
    \put(48.5,10){$-1$}
    \put(41,3){$-1$}
    \put(28,1){$2$}
    \put(57,1){$2$}
    \put(46.5,17.5){$2$}
    \put(78,4.5){$v_1$}
    \put(93,4.5){$v_2$}
    \put(82,10){$-\frac{\lambda+3}{\lambda-2}$}
    \put(66.5,7){$\frac{2\lambda-1}{\lambda-2}$}
    \put(100,7){$\frac{2\lambda-1}{\lambda-2}$}
    \end{overpic}
  \end{center}
  \caption{Reduction of the graph $L(G)$.}
\end{figure}

\begin{remark}
One could generalize $M_L(G)$ to any $G\in\mathbb{G}$ where $G$ has no loops and $n$ vertices by setting $M_L(G)_{ij}=-M(G)_{ij}$ for $i\neq j$ and $M_L(G)_{ii}=\sum_{j=1,j\neq i}^n M(G)_{ij}$. This generalizes and is consistent with what is done for weighted digraphs in \cite{Wu05} for example.
\end{remark}

\subsection{Sequential Reductions} 
As any reduction $\mathcal{R}_S(G)$ of a graph $G\in\mathbb{G}$ over the structural set $S$ is again a graph in $\mathbb{G}$ it is natural to consider sequences of reductions on a graph as well as to what degree a graph can be reduced via such reductions. In order to address this we need to first extend our notation to an arbitrary sequence of reductions. 

\begin{definition}
For $G=(V,E,\omega)$ suppose $S_1,\dots,S_m\subseteq V$ such that $S_1\in st(G)$, $\mathcal{R}_1(G)=\mathcal{R}_{S_1}(G)$ and $$S_{i+1}\in st(\mathcal{R}_i(G)) \ \text{where} \ \mathcal{R}_{S_{i+1}}(\mathcal{R}_i(G))=\mathcal{R}_{i+1}(G), \  1\leq i\leq m-1.$$ If this is the case then we say $S_1,\dots,S_m$ \textit{induces a sequence of reductions} on $G$ and we write $\mathcal{R}_i(G)=\mathcal{R}(G;S_1,\dots,S_i)$ for $1\leq i \leq m-1$. Moreover, we let $$\mathcal{N}(G;S_1,\dots,S_i)=\mathcal{N}(G;S_1,\dots,S_{i-1})\cup\mathcal{N}(\mathcal{R}_{i-1}(G);S_i), \ 2\leq i\leq m.$$ 
\end{definition}

The following is an immediate corollary of theorem \ref{theorem1}.

\begin{corollary}
Suppose $S_1,\dots,S_m$ induces a sequence of reductions on the graph $G\in\mathbb{G}$. Then $\sigma(G)$ and $\sigma(\mathcal{R}(G;S_1,\dots,S_m))$ differ at most by $\mathcal{N}(G;S_1,\dots, S_m)$.
\end{corollary}

\begin{definition}
For $p\in\mathbb{C}[\lambda]$ let $deg(p)$ be the degree of $p$ and for $\omega=p/q\in\mathbb{W}$ let $\pi(\omega)=deg(p)-deg(q)$. Let $\mathbb{G}_\pi\subseteq\mathbb{G}$ be the set of graphs with the property that for any $G\in\mathbb{G}_\pi$ with vertex set $\{v_1,\dots,v_n\}$, $\pi(M(G)_{ij})\leq 0$ for all $1\leq i,j\leq n$.
\end{definition}

\begin{remark}
Note that any graph $G$ where $M(G)\in\mathbb{C}^{n\times n}$ is in the set $\mathbb{G}_\pi$.
\end{remark}

With this in place we give the following theorem on sequential reductions.

\begin{theorem}\label{theorem3}\textbf{(Commutativity of Reductions)}
For $G\in\mathbb{G}_\pi$ suppose the sequences $S_1,\dots,S_m$ and $T_1,\dots,T_n$ both induce a sequence of reductions on $G$. If $S_m=T_n$ then $\mathcal{R}(G;S_1,\dots,S_m)=\mathcal{R}(G;T_1,\dots,T_n)$.
\end{theorem}

That is, the final vertex set in a sequence of reductions completely specifies the reduced graph irrespective of the specific sequence. However, if  $S_1,\dots,S_m$ and $T_1,\dots,T_n$ both induce a sequence of reductions on $G\in\mathbb{G}_\pi$ it is possible that $\mathcal{N}(G;S_1,\dots,S_m)\neq\mathcal{N}(G;T_1,\dots,T_n)$ even if $S_m=T_n$. A natural goal then in designing a sequence of reductions is to somehow minimize $\mathcal{N}(G;S_1,\dots,S_m)$ by carefully choosing the structural set $S_i$ at the $ith$ step in a particular sequence.

To address the extent to which a graph may be reduced via some sequence of reductions as well as indicate the uniqueness and flexibility of such reductions we give the following theorem.

\begin{theorem}\label{theorem-1}\textbf{(Existence and Uniqueness)}
Let $G=(V,E,\omega)$ and $\mathcal{V}$ be any nonempty subset of $V$. If $G\in\mathbb{G}_\pi$ then there exist a sequence $S_1,\dots,S_{m-1},\mathcal{V}$ inducing a sequence of reductions on $G$. Moreover, for any such sequence $T_1,\dots,T_{n-1},\mathcal{V}$ there is a unique graph $\mathcal{R}_{\mathcal{V}}[G]=\mathcal{R}(G;T_1,\dots,T_{n-1},\mathcal{V})$ independent of the particular sets $T_1,\dots,T_{m-1}$.
\end{theorem}

\begin{remark}
It is important to mention that $\mathcal{V}$ in theorem \ref{theorem-1} is any subset of vertices of the graph $G$. In particular, $\mathcal{V}$ may not be a structural set of $G$.
\end{remark}

\section{Reductions in Steps}

In this section we introduce a related but alternate way of describing the branch structure of a graph where this structure is again related to the graph's spectrum. We then use this structure along with a small number of graph transformations to give a constructive method for reducing any graph $G\in\mathbb{G}$ over any $S\in st(G)$.

\subsection{Branch Decompositions}
For a given branch $\beta=v_1,\dots,v_m$ in the graph $G=(V,E,\omega)$ let the \textit{weight sequence} of $\beta$ be the sequence given by $$\Omega(\beta)=\omega(e_{11}),\dots,\omega(e_{i-1,i}),\omega(e_{ii}),\omega(e_{i,i+1}),\dots,\omega(e_{mm}).$$ 

\begin{definition}
For a structural set $S=\{v_1,\dots,v_m\}$ of a digraph $G$ let $$\mathcal{D}_S(G)=\bigcup_{1\leq i,j \leq m}\{(v_i,v_j;\Omega(\beta)):\beta\in\mathcal{B}_{ij}(G;S)\}$$ be the \textit{branch decomposition} of $G$ with respect to $S$ where this set includes multiplicities. 
\end{definition}

Note that $\mathcal{D}_S(G)$ can be written as lists but the formulation above is more convenient.

\begin{definition}\label{def3}
Let $G,H\in\mathbb{G}$. We say $G$ and $H$ have a common branch decomposition via the structural sets $S$ and $T$ respectively if $S\in st(G)$, $T\in st(H)$ and there is a one-to-one map $\rho:S\rightarrow T$ such that $$D_T(H)=\bigcup_{1\leq i,j \leq m}\{(\rho(v_i),\rho(v_j);\Omega(\beta)):\beta\in\mathcal{B}_{ij}(G;S)\}.$$
\end{definition}

For example consider the unweighted digraphs $G$ and $H$ in Figures 1 and 3. For $T=\{v_1,v_4\}$ and $S=\{w_2,w_5\}$  
\begin{align*}
\mathcal{D}_T(H)=\big\{(v_1,v_1; 0,1,1,1,0)&,(v_1,v_4; 0,1,1,1,1),\\&(v_4,v_1; 1,1,0,1,0),(v_4,v_4; 1,1,0,1,1)\big\} 
\end{align*}
\begin{align*}
\mathcal{D}_S(G)=\big\{(w_2,w_2; 0,1,1,1,0)&,(w_2,w_5; 0,1,1,1,1),\\&(w_5,w_2; 1,1,0,1,0),(w_5,w_5; 1,1,0,1,1)\big\}.
\end{align*} These graphs have the same branch decomposition via the map $\rho:S\rightarrow T$ given by $\rho(v_2)=w_2$ and $\rho(v_3)=w_5$

Note that if two graphs $G$ and $H$ have a common branch decomposition with respect to the structural sets $S$ and $T$ respectively then there is some $\rho:S\rightarrow T$ such that $\mathcal{B}_{ij}(G;S)=\mathcal{B}_{\rho(i)\rho(j)}(H;T)$ for all $i,j$. This implies the following corollary of theorem \ref{theorem1}.

\begin{corollary}
Let $G,H\in\mathbb{G}$ having a common branch decomposition via $S$ and $T$ respectively. Then $\mathcal{R}_S(G)\simeq\mathcal{R}_T(H)$ and $\sigma(G)$ and $\sigma(H)$ differ at most by $\mathcal{N}(G;S)\cup\mathcal{N}(H;T)$.
\end{corollary}

That is, common branch decompositions imply common reductions. One of the useful distinctions between branch reductions and decompositions is that the branch decomposition of a graph contains the lengths of the individual branches whereas the reduction does not. 

Moreover, we note that in some ways branch decompositions are a more natural graph theoretic object since graphs with a common branch decomposition share the same weight set. It will be the fact that the branch decomposition of a graph retains the weight set and branch lengths that will allow us to prove theorem \ref{theorem1}.

\subsection{Branch Manipulations}

If $G\in\mathbb{G}$ and $S\in st(G)$ then two branches $\beta_1, \beta_2\in \mathcal{B}_S(G)$ are said to be \textit{independent} if they have no interior vertices in common.

\begin{definition}
Let $G,\mathcal{X}\in\mathbb{G}$ be graphs with a common branch decomposition both with respect to the same set
$S$. Then $\mathcal{X}$ is a \textit{branch expansion} of $G$ with respect to $S$ if any two $\beta_1,\beta_2\in\mathcal{B}_S(\mathcal{X})$ are independent and each vertex of $\mathcal{X}$ is on a branch of $\mathcal{B}_S(\mathcal{X})$. 
\end{definition}

Note that $G$ and its branch expansion $\mathcal{X}$ may have all vertices in common or share only the vertices in $S$.

\begin{lemma}\textbf{(Branch Expansions)}\label{lemma0}
Let $S\in\ st(G)$ for $G\in\mathbb{G}$. Then an expansion $\mathcal{X}$ of $G$ with respect to $S$ exists and $\sigma(G)$ and $\sigma(\mathcal{X})$ differ at most by $\mathcal{N}(G;S)$.
\end{lemma}

An example of a branch expansion is seen in figures 1 and 3, $G$ being an expansion of $H$ over the set $T=\{v_1,v_4\}$ if the vertices of $G$ are relabeled via $w_2\mapsto v_2$ and $w_5\mapsto v_3$.

Let $G=(V,E,\omega)$ be a weighted digraph where $V=\{v_1,\dots,v_n\}$ and $e_{ik}\in E$. Suppose $\omega_{ij},\omega_{jj},\omega_{jk}\in\mathbb{W}$ for $j\neq i,k$ where $\omega(e_{ij})=\omega_{ij}\omega_{jk}/(\lambda-\omega_{jj})$. If we replace $e_{ik}$ in $G$ by the two edges and loop with associated weights as in figure 5 then we call the resulting graph the graph $G$ with \textit{loop bisected edge} $e_{ik}$ with intermediate vertex $v_j$.

\begin{figure}
  \begin{center}
    \begin{overpic}[scale=.45]{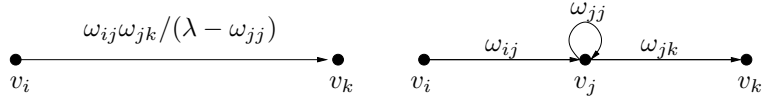}
    \put(10,4){$\omega_{ij}\omega_{jk}/(\lambda-\omega_{jj})$}
    \put(0,-3){$v_i$}
    \put(54,-3){$v_i$}
    \put(43,-3){$v_k$}
    \put(98,-3){$v_k$}
    \put(76,-3){$v_j$}
    \put(75.5,7){$\omega_{jj}$}
    \put(64,2){$\omega_{ij}$}
    \put(85,2){$\omega_{jk}$}
    \end{overpic}
  \end{center}
  \caption{Loop bisection of an edge.}
\end{figure}

\begin{lemma}\label{lemma2}\textbf{(Loop Bisection)}
Let $G=(V,E,\omega)$ where $e_{ik}\in E$ such that $\omega(e_{ik})=\omega_{ij}\omega_{jk}/(\lambda-\omega_{jj})$. If $\tilde{G}$ is the graph $G$ with loop bisected edge $e_{ik}$ and $S\in st(G)$ then $S\in st(\tilde{G})$ and the spectra $\sigma(G)$, $\sigma(\tilde{G})$ differ at most by $\mathcal{N}(G;S)$.
\end{lemma}

\subsection{Independent Branch Reductions}
The goal now is to combine lemmas \ref{lemma0} and \ref{lemma2} to construct the reduction $\mathcal{R}_S(G)$ which can be done as follows.

The first step in the reduction of $G$ over the structural set $S$ is to in fact do the opposite. That is, we first would like to find a branch expansion of $G$. Note that we can explicitly construct an expansion $\mathcal{X}$ of $G$ via $\mathcal{D}_S(G)$ by taking each pair of vertices in $S$ and connecting them by independent branches with weight sets and multiplicities as specified in this decomposition. Since each vertex of $\mathcal{X}$ is by construction on some branch of $\mathcal{X}$ then this is in fact an expansion of $G$ giving a proof to the existence claim in lemma \ref{lemma0}.

For the next step in this reduction we use the lemma \ref{lemma2} to shorten the lengths of the independent branches in the expanded graph $\mathcal{X}$. Specifically, if $\beta\in\mathcal{B}_{ij}(\mathcal{X},S)$ has weight set $\Omega(\beta)=\{\omega_1,\ell_1,\omega_2,\ell_2\dots,\omega_{n-1},\ell_{n-1},\omega_n\}$, $n>1$ then by lemma \ref{lemma2} we may modify this weight set to $\{\omega_1\omega_2/(\lambda-\ell_1),\ell_2,\dots,\omega_{n-1},\ell_{n-1},\omega_n\}$ without effecting the spectrum of the graph by more than $\mathcal{N}(G;S)$. If this is continued until $\beta$ is reduced to a single edge $\beta_e$ from $v_i$ to $v_j$ then $\beta_e$ has weight $\mathcal{P}_{\omega}(\beta)$.

If every branch of $\mathcal{X}$ is contracted to a single edge in this way then, after the multiple edges are made single via our convention, the resulting graph is the graph $\mathcal{R}_S(G)$ defined in definition \ref{reductiondef}. Moreover, as each step in this process does not change the spectrum of the graph by more than $\mathcal{N}(G;S)$ then the proof of theorem \ref{theorem1} follows once lemmas \ref{lemma0} and \ref{lemma2} are known to hold.

\section{Proofs}
Before we prove the main results of this paper we note the following (see \cite{Brualdi91} for details). First, a directed graph is \textit{strongly connected} if there is a path (possibly of length zero) from each vertex of the graph to every other vertex. The \textit{strongly connected components} of $G=(V,E)$ are its maximal strongly connected subgraphs. Moreover, its vertex set $V=\{v_1,\dots,v_n\}$ can always be labeled in such a way that $M(G)$ has the following triangular block structure

$$M(G)=\left[ \begin{array}{cccc}
M(\mathbb{S}_1)      & 0     & \dots  & 0\\
*        & M(\mathbb{S}_2)   &        & \vdots\\
\vdots   &       & \ddots & 0\\
*        & \dots & *      & M(\mathbb{S}_m) 
\end{array}
\right]$$ where $\mathbb{S}_i$ is a strongly connected component of $G$ and $*$ are block matrices with possibly nonzero entries. As $\det(M(G))=\prod_{i=1}^m\det(M(\mathbb{S}_i))$ then, since edges between strongly connected components correspond to the entries in the block matrices below the diagonal blocks, these edges may be removed or their weights changed without effecting $\sigma(G)$. Moreover, as an edge of $G$ is in a strongly connected component of $G$ if and only if it is on some cycle then all edges belonging to no strongly connected components of $G$ can be removed without effecting $\sigma(G)$.

With this in mind, for the graph $G=(V,E,\omega)$ where $V=\{v_1,\dots,v_n\}$ and $S\in st(G)$, we say an edge $e_{ij}\in E$ is not on any branch of $\mathcal{B}_S(G)$ if $v_i,v_j$ do not both belong to some $\beta$ for all $\beta\in\mathcal{B}_S(G)$. If $e_{ij}$ is not on any branch of $\mathcal{B}_S(G)$ then the claim is that it cannot be on a cycle unless the cycle is a loop. 

To see this note that every cycle which is not a loop must contain a vertex of $S$ for $\bar{S}$ not to induce a cycle in $\ell(G)$. Hence, every cycle is either a single branch or the union of several branches in $\mathcal{B}_S(G)$. This implies that all edges except for loops off the branch set $\mathcal{B}_S(G)$ may be removed without effecting the graph's spectrum.

On the other hand, suppose $e_{jj}$ is a loop, having weight possibly equal to 0, on the vertex $v_j$ where $v_j$ is not a vertex on any branch of $\mathcal{B}_S(G)$. Then this vertex may also be removed from $G$ without effecting the spectrum of the graph by more than $\mathcal{N}(G;S)$. This follows from the fact in this case that the vertex $v_j$ is itself a strongly connected component of $G$ since, by the discussion above, it lies on no other cycle of $G$. Hence, if $\mathbb{S}_1,\dots,\mathbb{S}_m$ are the strongly connected components of $G$ where $v_j=\mathbb{S}_\ell$ then 
\begin{equation}\label{eq3}
\det(M(G)-\lambda I)=(\omega(e_{jj})-\lambda)\prod_{k=1,k\neq \ell}^m\det(M(\mathbb{S}_k)-\lambda I).
\end{equation} Therefore, removing $v_j$ from $G$ (which removes $e_{jj}$) changes $\sigma(G)$ at most by solutions to the equation $\lambda=\omega(e_{jj})$ or the $\lambda$ such that $\omega(e_{jj})$ is undefined, all of which are in $\mathcal{N}(G;S)$. We record this as the following proposition.

\begin{proposition}\label{prop1}
Let $G=(V,E,\omega)$ and $S\in st(G)$ where $V_S,E_{S}$ are the set of vertices and edges respectively not on any branch of $\mathcal{B}_S(G)$. If $\mathcal{G}=(V\setminus V_{S},E\setminus E_{S},\omega)$ then the spectra $\sigma(G)$ and $\sigma(\mathcal{G})$ differ at most by $\mathcal{N}(G;S)$.
\end{proposition}

For ease of notation we adopt the following. For a square matrix $M$ denote by $[M]_{ij}$ the \textit{minor} of $M$ given by the determinant of the reduced matrix formed by omitting the $i$th row and $j$th column of $M$. Continuing in this manner we let $[[M]_{ij}]_{kl}=[M]_{ij,kl}$ be the determinant when row $i$, column $j$ are omitted then row $k$, column $j$ and so on. We now give a proof of lemma \ref{lemma0}. 

\begin{proof}
Let $H=(V,E,\omega)$ where $V=\{v_1,\dots,v_n\}$, $n\geq 3$ and $e_{1i},e_{i1}\notin E$ for $2\leq i\leq n$. Also let $\omega(e_{11})=\omega(e_{22})$, and $e_{32}\in E$. From $H$ we construct the graph $\tilde{H}$ first by switching the edge $e_{32}$ to $e_{31}$ while maintaining its edge weight. Second, if $\mathcal{O}(v_i)=\{v_j\in V:e_{ij}\in E, j\neq i\}$ make $\mathcal{O}(v_1)$ the same as $\mathcal{O}(v_2)$ such that $\omega(e_{1i})=\omega(e_{2i})$ for all $3 \leq i \leq n$. Let the resulting graph be the graph $\tilde{H}$ (see figure 6).

\begin{figure}
  \begin{center}
    \begin{overpic}[scale=.45]{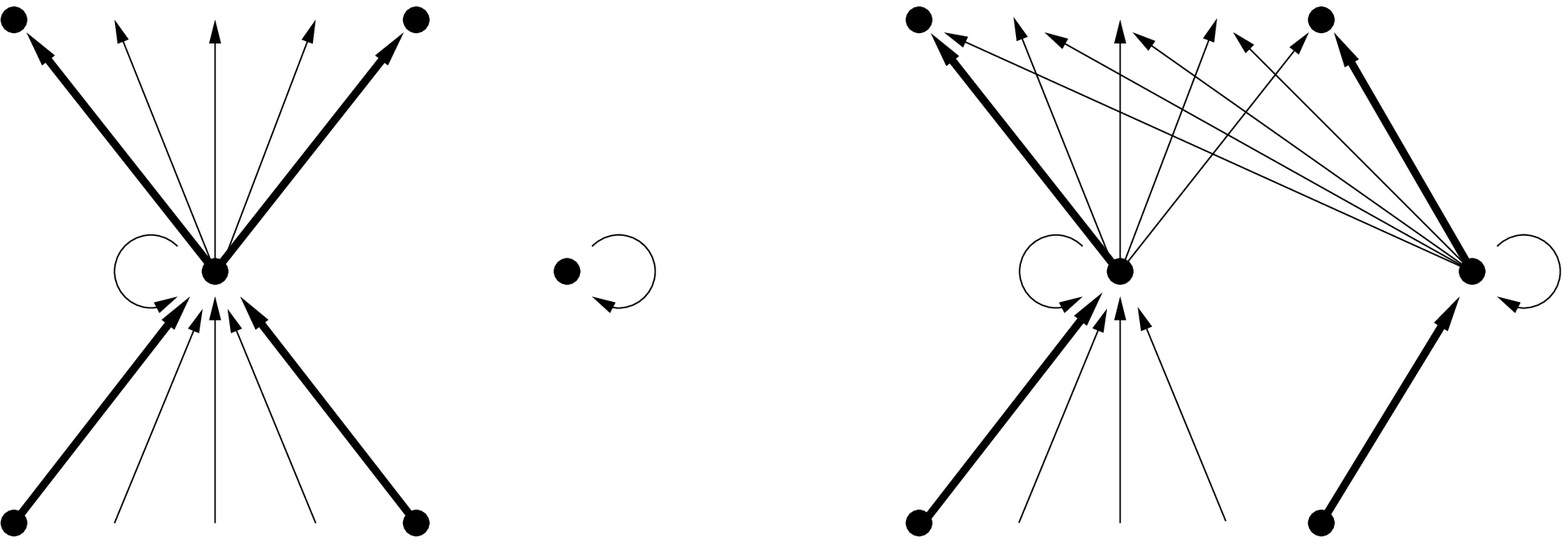}
    \put(14,-5){$H$}
    \put(75,-5){$\tilde{H}$}
    \put(30,16){$v_1$}
    \put(87,16){$v_1$} 
    \put(18,16){$v_2$}
    \put(75,16){$v_2$}
    \put(30,0){$v_3$}
    \put(91,0){$v_3$}
    \put(-5,0){$v_4$}
    \put(53,0){$v_4$}
    \put(0,25){$\beta_2$}
    \put(24,25){$\beta_1$}
    \put(58,25){$\beta_2$}
    \put(92,25){$\beta_1$}
    \end{overpic}
  \end{center}
  \caption{Separation of branches $\beta_1$ and $\beta_2$.}
\end{figure}

If $M=M(H)-\lambda I$ and $\tilde{M}=M(\tilde{H})-\lambda I$ the claim is that $\det(M)=\det(\tilde{M})$. To see this note that $[\tilde{M}]_{31,22}=[\tilde{M}]_{11,32}$ and as well that $[\tilde{M}]_{31,i2}=0$ for all $4\leq i \leq n$ since the first two rows in the associated matrices are identical. If $\omega(e_{ij})=\omega_{ij}$ for all $e_{ij}\in E$ then
\begin{align*}
&\det(\tilde{M})=(\omega_{22}-\lambda)[\tilde{M}]_{11}+\omega_{32}[\tilde{M}]_{31}=\\
&(\omega_{22}-\lambda)\big((\omega_{22}-\lambda)[\tilde{M}]_{11,22}+\sum_{i=4}^n(-1)^i \omega_{i2}[\tilde{M}]_{11,i2}\big)-\omega_{32}(\omega_{22}-\lambda)[\tilde{M}]_{31,22}=\\
&(\omega_{22}-\lambda)\big((\omega_{22}-\lambda)[\tilde{M}]_{11,22}+\sum_{i=3}^n(-1)^i \omega_{i2}[\tilde{M}]_{11,i2}\big).
\end{align*}
Since $[\tilde{M}]_{11,i2}=[M]_{11,i2}$ for $2 \leq i \leq n$ then 
\begin{align*}
&\det(M)=(\omega_{22}-\lambda)[M]_{11}=\\
&(\omega_{22}-\lambda)\big((\omega_{22}-\lambda)[M]_{11,22}+\sum_{i=3}^n(-1)^i \omega_{i2}[M]_{11,i2}\big)=\det(\tilde{M}).
\end{align*}
Hence, $\sigma(H)=\sigma(\tilde{H})$. Moreover, it follows that if $S\in st(H)$ where $v_1,v_2\notin S$ then $S\in st(\tilde{H})$ and both $H$ and $\tilde{H}$ have the same branch decomposition with respect to $S$.

To use this construction to our advantage let $\mathcal{G}=(\mathcal{V},E,\omega)$ such that $\mathcal{V}=\{v_2,\dots,v_n\}$, $n\geq 3$. If $\mathcal{G}_1=(\{v_1\},\{e_{11}\},\mu)$, $\mu(e_{11})=\omega(e_{22})$, and $G=\mathcal{G}_1\cup\mathcal{G}$ then for any $S\in st(\mathcal{G})$ where $v_2\notin S$ it follows that $S\in st(G)$. 

Suppose then that $\beta_1,\beta_2\in\mathcal{B}_S(G)$ are distinct but not independent where their first interior point of intersection is the vertex $v_2$. If it is the case that the vertex preceding $v_2$ on $\beta_1$ is $v_3$ and the vertex preceding $v_2$ on $\beta_2$ is $v_4$ where $v_3\neq v_4$ then $G$ has the structure of $H$ in the figure 6. Applying the same procedure to $G$ as we did to $H$ we obtain the graph $\tilde{G}$ having the same spectrum as $G$ where $S\in st(\tilde{G})$ such that both graphs have the same branch decomposition with respect to $S$. Moreover, the branches $\beta_1$ and $\beta_2$ in $\tilde{G}$ no longer meet at $v_2$ but at some vertex following $v_2$, if at all.  

\begin{figure}
  \begin{center}
    \begin{overpic}[scale=.45]{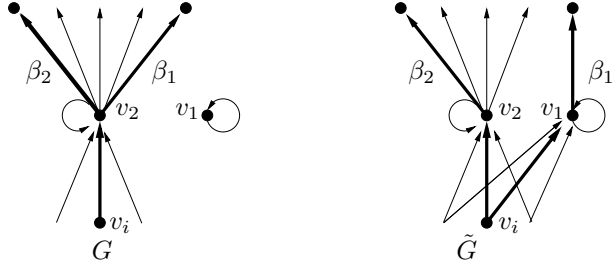}
    \put(14,-5){$G$}
    \put(75,-5){$\tilde{G}$}
    \put(28,19){$v_1$}
    \put(89,19){$v_1$} 
    \put(18,19){$v_2$}
    \put(82,19){$v_2$}
    \put(17,0){$v_i$}
    \put(82,0){$v_i$}
    \put(3,25){$\beta_2$}
    \put(24,25){$\beta_1$}
    \put(67,25){$\beta_2$}
    \put(97,25){$\beta_1$}
    \end{overpic}
  \end{center}
  \caption{Separation of branches $\beta_1$ and $\beta_2$.}
\end{figure}

If the vertex $v_3=v_4$, then under the assumption that $\beta_1\neq\beta_2$, by a relabeling of the vertices in $G$, with the exception of $v_1$, let $v_2$ be the first vertex on $\beta_1$ and $\beta_2$ where the following vertex on these branches differ. If $G^T$ is the graph with adjacency matrix $M(G)^T$ (where $T$ denotes transpose) then $\sigma(G)=\sigma(G^T)$ and $S\in st(G^T)$. Moreover, the branches $\beta_1^T,\beta_2^T\in\mathcal{B}_S(G^T)$ corresponding to $\beta_1,\beta_2\in\mathcal{B}_S(G)$ can be separated at $v_2$ again using the method above. If this modified graph is denoted $\tilde{G}^T$ then the graph $\tilde{G}$ with adjacency matrix $M(\tilde{G}^T)^T$ (see figure 7) and $G$ again have the same spectrum as well as having the same branch decomposition with respect to $S$. As before, $\beta_1$ and $\beta_2$ in $\tilde{G}$ meet at some point following $v_2$, if at all.

Since the subgraph $\mathcal{G}_1$ of $G$ is a strongly connected component with single vertex $v_1\notin S$ then proposition \ref{prop1} implies $\sigma(\mathcal{G})$ and $\sigma(G)$ differ at most by $\mathcal{N}(G;S)$ which inturn implies the same for $\sigma(\mathcal{G})$ and $\sigma(\tilde{G})$. Let $\tilde{\mathcal{X}}$ denote the graph given by repeating this process over all branches in $\mathcal{B}_S(G)$ until each is independent of all other branches. Then $D_S(\mathcal{G})=D_S(\tilde{X})$ and $\sigma(\mathcal{G})$ and $\sigma(\tilde{\mathcal{X}})$ differ at most by $\mathcal{N}(G;S)$. 

Again proposition \ref{prop1} can be used to remove any vertex or edge not belonging to a branch of $\mathcal{B}_S(\tilde{\mathcal{X}})$ without effecting the spectrum of $\tilde{\mathcal{X}}$ by more than $\mathcal{N}(G;S)$. The removal of these edges and vertices produces the graph $\mathcal{X}$ where $\mathcal{X}$ is a branch expansion of $\mathcal{G}$ with the desired properties.
\end{proof}

We now give a proof of lemma \ref{lemma2}.

\begin{proof} We proceed as in the previous proof. Let $\mathcal{G}=(V,E,\omega)$ have vertex set $V=\{v_2,\dots,v_n\}$, $n\geq 3$ where $e_{23}$ is an edge with weight $\omega_{21}\omega_{23}/(\lambda-\omega_{11})$. Let $\tilde{G}$ be the graph $\mathcal{G}$ with loop bisected edge $e_{23}$ with intermediate vertex $v_1$. That is, $\tilde{G}$ has vertex set $\{v_1,\dots,v_n\}$ where $\omega(e_{11})=\omega_{11}$, $\omega(e_{21})=\omega_{21}$, and $\omega(e_{13})=\omega_{13}$. 

The claim is that $\det(M(\mathcal{G})-\lambda I)=(\omega_{11}-\lambda)\det(M(\tilde{G})-\lambda I)$. To see this let $G=\mathcal{G}_1\cup\mathcal{G}$ where $\mathcal{G}_1=(\{v_1\},\emptyset)$ as well as $\omega(e_{ij})=\omega_{ij}$ and consider the two matrices $M=M(G)-\lambda I$ and $\tilde{M}=M(\tilde{G})-\lambda I$. Assuming $v_2\neq v_3$ 
\begin{align*}
&-\lambda(\det(\tilde{M}))=-\lambda\big((\omega_{11}-\lambda)[\tilde{M}]_{11}-\omega_{21}[\tilde{M}]_{21}\big)=\\
&-\lambda\big((\omega_{11}-\lambda)\big((\omega_{22}-\lambda)[\tilde{M}]_{11,22}+\sum_{i=4}^n(-1)^i\omega_{2i}[\tilde{M}]_{11,2i}\big)-\omega_{21}(-\omega_{13}[\tilde{M}]_{21,13})\big).
\end{align*}
On the other hand,
\begin{align*}
&(\omega_{11}-\lambda)\det(M)=(\omega_{11}-\lambda)(-\lambda)[M]_{11}=\\
&-\lambda(\omega_{11}-\lambda)\big((\omega_{22}-\lambda)[M]_{11,22}-\frac{\omega_{21}\omega_{13}}{\omega_{11}-\lambda}[M]_{11,23}+\sum_{i=4}^n(-1)^i\omega_{2i}[M]_{11,2i}\big)=\\
&-\lambda\big((\omega_{11}-\lambda)(\omega_{22}-\lambda)[M]_{11,22}+\omega_{21}\omega_{13}[M]_{11,23}+(\omega_{11}-\lambda)\sum_{i=4}^n(-1)^i\omega_{2i}[M]_{11,2i}\big).
\end{align*}
As $[\tilde{M}]_{21,13}=[\tilde{M}]_{11,23}$ and $[\tilde{M}]_{11,2i}=[M]_{11,2i}$ for $i\geq 2$ then $(\omega_{11}-\lambda)\det(M)=-\lambda\det(\tilde{M})$. Since $-\lambda\det(M(\mathcal{G})-\lambda I)=\det(M)$ then this implies $$-\lambda(w_{11}-\lambda)\det(M(\mathcal{G})-\lambda I)=-\lambda\det(\tilde{M})$$ verifying the claim as $-\lambda$ can be removed by multiplying by its inverse in $\mathbb{W}$.

An argument similar to that in the proof of proposition \ref{prop1} (see equation (\ref{eq3})) implies $\sigma(\mathcal{G})$ and $\sigma(\tilde{G})$ differ at most by $\mathcal{N}(G;S)$. And if it is the case that $v_2=v_3$ a similar argument implies the same result completing the proof.
\end{proof}

We note that the preceding two proofs are enough to prove theorem \ref{theorem1} via the method described in section 4.3. We now turn to the proof of theorem \ref{theorem3} for which we will need the following.

\begin{lemma}\label{lemma-1}
If $G=(V,E,\omega)$ is a graph in $\mathbb{G}_\pi$ where $S\in st(G)$ then $\mathcal{R}_S(G)\in\mathbb{G}_\pi$. Moreover, if $V=\{v_1,\dots,v_n\}$ where $n\geq 2$ then $V\setminus\{v_i\}\in st(G)$ for all $1\leq i\leq n$.
\end{lemma}

\begin{proof}
Let $w_1=p_1/q_1$, $w_2=p_2/q_2$, and $w_3=p_3/q_3$ be elements of $\mathbb{W}$ such that $\pi(w_1),\pi(w_2)$, $\pi(w_3)\leq 0$. As $$\pi(w_1+w_2)=\pi\Big(\frac{p_1q_2+p_2q_1}{q_1q_2}\Big)\leq\max\{\pi(w_1),\pi(w_2)\},$$
$$\pi(w_1w_2)=\pi\Big(\frac{p_1p_2}{q_1q_2}\Big)=\pi(w_1)+\pi(w_2), \ \text{and}$$
$$\pi\Big(\frac{w_1w_2}{\lambda-w_3}\Big)=\pi\Big(\frac{p_1p_2q_3}{q_1q_2(q_3\lambda-p_3)}\Big)<\pi(w_1)+\pi(w_2)$$ are each nonpositive then the first statement of lemma \ref{lemma-1} follows from definition \ref{branchprod} and \ref{reductiondef}.

Consider then a single vertex $v_i$ of $G=(V,E,\omega)$ where $G$ is assumed to be in $\mathbb{G}_\pi$. As $v_i$ cannot induce any cycles in $\ell(G)$ and $\omega(e_{ii})\neq\lambda$ given $\pi(\omega(e_{ii}))\leq 0$ then $V\setminus\{v_i\}\in st(G)$ so long as $V$ has more than one vertex. The result then follows.
\end{proof}

Let $G=(V,E,\omega)$ be a graph in $\mathbb{G}_\pi$ with $V=\{v_1,\dots,v_n\}$. Taking the perspective that reductions amount to removing vertices from a graph then, by way of notation, if $\bar{\mathcal{R}}(G;v_i)=\mathcal{R}_{V\setminus\{v_i\}}(G)$ then inductively define $$\bar{\mathcal{R}}(G;v_1,\dots,v_i)=\bar{\mathcal{R}}(\bar{\mathcal{R}}(G;v_1,\dots,v_{i-1});v_i) \ \text{for} \ 1<i< n.$$
Note that the graph $\bar{\mathcal{R}}(G;v_1,\dots,v_i)$ is well defined by lemma \ref{lemma-1} for each $1<i< n$.

\begin{lemma}\label{lemma4}
Let $G=(V,E,\omega)$ where $G\in\mathbb{G}_\pi$. If $S\in st(G)$, $V=\{v_1,\dots,v_n\}$, and $\bar{S}=\{v_k,\dots,v_n\}$ where $2\leq k\leq n$ then for any permutation $\rho$ on $\{k,\dots,n\}$, $\mathcal{R}_S(G)=\bar{\mathcal{R}}(G;v_{\rho(k)},\dots,v_{\rho(n)})$.
\end{lemma} That is, any reduction over a structural set $S$ can be accomplished by sequentially removing vertices in $\bar{S}$ in any order without effecting the resulting graph. 

In the proof of this lemma we use the following. If $\alpha,\gamma\in\mathcal{B}_S(G)$ where $\alpha=v_1,\dots,v_m$ and $\gamma=v_m,\dots,v_n$ then let $\beta=\alpha*\gamma$ be the \textit{concatenation} of these branches given by $\alpha*\gamma=v_1,\dots,v_m,\dots,v_n$. Note that the branch $\emptyset\in\mathcal{B}_S(G)$ is the \textit{empty branch} with branch product 0. We now give a proof of lemma \ref{lemma4}.

\begin{proof}
Let $G$ and $S$ be as in the statement of lemma \ref{lemma4}. For the sake of notation let $S\cup\{v_k\}=S_k$ and note that as $\bar{S}$ induces no cycles in $\ell(G)$ then the same is true of $\bar{S_k}$ implying $S_k\in st(G)$. From this and lemma \ref{lemma-1} it then follows that $S\in st(\mathcal{R}_{S_k}(G))$. Therefore, let $\mathcal{R}_{S_k}(G)=(S_k,\tilde{\mathbb{E}},\tilde{\nu})$, $\mathcal{R}_S(\mathcal{R}_{S_k}(G))=(S,\mathbb{E},\nu)$, and $\mathcal{R}_S(G)=(S,\mathcal{E},\mu)$. Also for any graph $H=(\mathcal{V},F,\tau)$ with structural set $T$ let $\mathcal{B}_{ij}^k(H;T)$ be the branches in $\mathcal{B}_{ij}(H;T)$ containing the vertex $v_k\in \mathcal{V}$. 

Then for $1\leq i,j<k$ note that $\mathcal{B}_{ij}(\mathcal{R}_{S_k}(G);S)$ consists of at most two branches $\beta_1,\beta_2$ where $\beta_1\in\mathcal{B}_{ij}^k(\mathcal{R}_{S_k}(G),S) \ \text{and} \ \beta_2\notin\mathcal{B}_{ij}^k(\mathcal{R}_{S_k}(G),S)$ where either branch is possibly empty. Hence, $\nu(e_{ij})=\mathcal{P}_{\tilde{\nu}}(\beta_1)+\mathcal{P}_{\tilde{\nu}}(\beta_2)$.

In what follows let $\mathcal{B}^2(G,S_k)$ be the direct product $\mathcal{B}_{ik}(G,S_k)\times\mathcal{B}_{kj}(G,S_k)$. As $\beta_1=v_i,v_k,v_j$ and $\beta_2=v_i,v_j$ if each are nonempty then 
$$\mathcal{P}_{\tilde{\nu}}(\beta_1)=\displaystyle{\sum_{(\alpha,\gamma)\in\mathcal{B}^2(G,S_k)}\frac{\mathcal{P}_{\omega}(\alpha)\mathcal{P}_{\omega}(\gamma)}{\lambda-\omega(e_{kk})}}, \ \ \ \ \  \mathcal{P}_{\tilde{\nu}}(\beta_2)=\displaystyle{\sum_{\beta\notin\mathcal{B}_{ij}^k(G,S_k)}\mathcal{P}_{\omega}(\beta)}.$$
On the other hand, for $1\leq i,j<k$
$$\mu(e_{ij})=\displaystyle{\sum_{\beta\in\mathcal{B}_{ij}^k(G,S)}\mathcal{P}_{\omega}(\beta)+\sum_{\beta\notin\mathcal{B}_{ij}^k(G,S)}\mathcal{P}_{\omega}(\beta)}.$$
Note that if $\beta\in\mathcal{B}_{ij}^k(G,S)$ then there exists a unique pair $(\alpha,\gamma)\in\mathcal{B}^2(G,S_k)$ such that $\alpha*\gamma=\beta$. Hence, $\mathcal{P}_\omega(\beta)=\mathcal{P}_{\omega}(\alpha)\mathcal{P}_{\omega}(\gamma)/(\lambda-\omega(e_{kk}))$. Moreover, as 
$$\sum_{\beta\notin\mathcal{B}_{ij}^k(G,S)}\mathcal{P}_{\omega}(\beta)=\sum_{\beta\notin\mathcal{B}_{ij}^k(G,S_k)}\mathcal{P}_{\omega}(\beta)$$ then $\mu(e_{ij})=\nu(e_{ij})$ for all $1\leq i,j<k$ implying 
\begin{equation}\label{eq5}
\displaystyle{\mathcal{R}_S(G)=\bar{\mathcal{R}}(\mathcal{R}_{S\cup\{v_k\}}(G);v_k)}.
\end{equation}

To see how this is useful let $S\cup\{v_k,\dots,v_i\}=S_i$ for $k\leq i\leq n$. As $S_{k+1}=S_k\cup\{v_{k+1}\}\in st(G)$ for the same reason $S_k\in st(G)$ then by adjusting the indices in equation (\ref{eq5}) it follows that
$$\mathcal{R}_{S_k}(G)=\bar{\mathcal{R}}(\mathcal{R}_{S_{k+1}}(G);v_{k+1}).$$
By inserting this back into (\ref{eq5}), $\mathcal{R}_S(G)=$
$$\bar{\mathcal{R}}(\bar{\mathcal{R}}(\mathcal{R}_{S_{k+1}}(G);v_{k+1});v_k)=\bar{\mathcal{R}}(\mathcal{R}_{S_{k+1}}(G);v_{k+1},v_k).$$
Continuing in this way we arrive at
$$\mathcal{R}_S(G)=\bar{\mathcal{R}}(\mathcal{R}_{S_{n}}(G);v_n,\dots,v_k)=\bar{\mathcal{R}}(G;v_n,\dots,v_k)$$
where the final equality follows from the fact that $S_{n}=V$. As the labeling of $\bar{S}$ was arbitrary then $\mathcal{R}_S(G)=\bar{\mathcal{R}}(G;v_{\rho(k)},\dots,v_{\rho(n)})$ for any permutation $\rho$ on $\{k,\dots,n\}$.
\end{proof} 

\begin{lemma}\label{lemma5}
Let $G\in\mathbb{G}_\pi$ with vertex set $V=\{v_1,\dots,v_n\}$ and $n>2$. For any permutation $\rho$ on $\{1,\dots,i\}$ where $1\leq i<n$, $\bar{\mathcal{R}}(G;v_{1},\dots,v_{i})=\bar{\mathcal{R}}(G;v_{\rho(1)},\dots,v_{\rho(i)})$.
\end{lemma}

\begin{proof}
Given $G=(V,E,\omega)$ let $\bar{\mathcal{R}}(G;v_1)=(V_1,\tilde{E}_1,\tilde{\nu}_1)$, $\bar{\mathcal{R}}(G;v_1,v_2)=(V_{12},E_1,\nu_1)$, and $\bar{\mathcal{R}}(G;v_2,v_1)=(V_{12},E_2,\nu_2)$. For $1\leq k,l \leq n$ let $\omega(e_{kl})=\omega_{kl}$ and note that if $3\leq i,j \leq n$ where $i\neq j$ then 
$$\tilde{\nu}_1(e_{ij})=\omega_{ij}+\frac{\omega_{i1}\omega_{1j}}{\lambda-\omega_{11}}, \ \tilde{\nu}_1(e_{i2})=\omega_{i2}+\frac{\omega_{i1}\omega_{12}}{\lambda-\omega_{11}}, \ \text{and} \ \tilde{\nu}_1(e_{2j})=\omega_{2j}+\frac{\omega_{21}\omega_{1j}}{\lambda-\omega_{11}}.$$
This  implies in particular that $\nu_1(e_{ij})=$

\begin{equation*}\label{eq6}
\omega_{ij}+\big(\frac{\omega_{i1}\omega_{1j}}{\lambda-\omega_{11}}+\frac{\omega_{i2}\omega_{2j}}{\lambda-\omega_{22}}\big)+\big(\frac{\omega_{i2}\omega_{21}\omega_{1j}}{(\lambda-\omega_{11})(\lambda-\omega_{22})}+\frac{\omega_{i1}\omega_{12}\omega_{2j}}{(\lambda-\omega_{11})(\lambda-\omega_{22})}\big).
\end{equation*}

As this equation is symmetric in the indices 1 and 2 then $\nu_1(e_{ij})=\nu_2(e_{ij})$ for $3\leq i,j \leq n$, $i\neq j$. If $i=j$ then it follows that $\nu_1(e_{ij})=\omega_{ij}=\nu_2(e_{ij})$. Hence, $\bar{\mathcal{R}}(G;v_1,v_2)=\bar{\mathcal{R}}(G;v_2,v_1)$. Note that this can be used to show
$$\bar{\mathcal{R}}(G;v_1,v_2,\dots,v_j,v_{j+1},\dots,v_i)=\bar{\mathcal{R}}(G;v_1,v_2\dots,v_{j+1},v_j,\dots,v_i), \ 1\leq j<i$$
or that any transposition in the order of the removal of vertices $v_j,v_{j+1}$ does not effect the reduction. As any permutation $\pi$ on $\{1,\dots,i\}$ can be achieved via some number of transpositions it follows that $\bar{\mathcal{R}}(G;v_{1},\dots,v_{i})=\bar{\mathcal{R}}(G;v_{\pi(1)},\dots,v_{\pi(i)})$.
\end{proof}

We now give a proof of theorem \ref{theorem3}.

\begin{proof}
Let $S_1,\dots,S_m$ and $T_1,\dots,T_n$ both induce a sequence of reductions on the graph $G\in\mathbb{G}$ where $S_m=T_n$. If $\tilde{S}_i=S_i\setminus S_{i+1}$ and $\tilde{T}_j=T_j\setminus T_{j+1}$ then label $\tilde{S}_i=\{s_1^i,\dots,s_{k_i}^i\}$, $\tilde{T}_j=\{t_1^j,\dots,t_{l_j}^j\}$ where $k_i,l_j\in\mathbb{N}$ for $1\leq i\leq m$, $1\leq j\leq n$. 

By repeated use of lemma \ref{lemma4} it follows that 
\begin{align*}
\mathcal{R}(G;S_1,\dots,S_m)=&\bar{\mathcal{R}}(G;s^1_1,s^1_2,\dots,s_{k_1}^1,s_1^2,s_2^2,\dots,s^m_{k_m}) \ \text{and}\\ \mathcal{R}(G;T_1,\dots,T_n)=&\bar{\mathcal{R}}(G;t^1_1,t^1_2,\dots,t_{l_1}^1,t_1^2,t_2^2,\dots,t^n_{l_n})
\end{align*}
since the order in which the vertices of $\tilde{S}_i$ and $\tilde{T}_j$ are removed does not matter so long as the vertices in $\tilde{S}_i$ are removed before those in $\tilde{S}_{i+1}$ and the vertices in $\tilde{T}_j$ before $\tilde{T}_{j+1}$.

Since $\bigcup_{i=1}^m\tilde{S}_i=\bigcup_{j=1}^n\tilde{T}_j$ then under the relabeling of $s^1_1,s^1_2,\dots,s^m_{k_m}$ to $v_1,\dots,v_N$ let $\pi$ be the permutation on $\{1,\dots,N\}$ such that $v_{\pi(1)},\dots,v_{\pi(N)}=t^1_1,\dots,t^n_{l_n}$. Lemma \ref{lemma5} then implies 
$\bar{\mathcal{R}}(G;s^1_1,s^1_2,\dots,s^m_{k_m})=\bar{\mathcal{R}}(G;t^1_1,t^1_2,\dots,t^n_{l_n})$. It then follows that $\mathcal{R}(G;S_1,\dots,S_m)=\mathcal{R}(G;T_1,\dots,T_n)$.
\end{proof}

For a proof of theorem \ref{theorem-1} we give the following.

\begin{proof}
Let $G=(V,E,\omega)$ be a graph in $\mathbb{G}_\pi$ where $\mathcal{V}\subseteq V$ is any nonempty set. If $\bar{\mathcal{V}}=\{v_1,\dots,v_m\}$ then lemma \ref{lemma-1} implies that $\bar{R}(G;v_1,\dots,v_m)$ is well defined completing the proof of the first statement of the theorem.

The proof of the second statement follows by a direct application of theorem \ref{theorem3}.
\end{proof}

\section{Reductions and Strongly Connected Components}
Recall from section 5 that the strongly connected components of a graph $G\in\mathbb{G}$ carry all the information necessary to compute the spectrum of $\sigma(G)$. A natural question then is to what degree graph reductions respect the structure of strongly connected components. Our first result in this direction is the following.

\begin{lemma}\label{lemma6}
For the weighted digraph $G=(V,E,\omega)$ suppose $S\in st(G)$ and $\mathbb{S}_1,\dots,\mathbb{S}_m$ are the strongly connected components of $G$. If $S_i$ is the nonempty set of vertices of $\mathbb{S}_i$ contained in $S$ then $\mathcal{R}_{S_i}(\mathbb{S}_i)$ are strongly connected components of $\mathcal{R}_S(G)$ for $1\leq i\leq m$ where $S_i\neq\emptyset$.
\end{lemma}

\begin{proof}
Suppose $v_i,v_j$ are vertices in $S_k$ which are possibly nondistinct. Assuming $\mathbb{S}_k$ has more than one vertex then there is a path or cycle $P$, with positive length, contained in $\mathbb{S}_k$ from $v_i$ to $v_j$. As every edge in this strongly connected component must be on a cycle then $P=\beta_1*\dots *\beta_n$ where each $\beta_m\in\mathcal{B}_S(\mathbb{S}_k)$ for $1\leq m\leq n$ since every cycle in $G$ contains at least one vertex in $S$. If $b_m\in S$ is the vertex joining $\beta_{m-1}$ and $\beta_m$ there is a path or cycle $v_i,b_1,,b_2\dots,b_n,v_j$ in $\mathcal{R}_S(G)$ from $v_i$ to $v_j$. Switching the roles of $v_i$ and $v_j$ it follows that both the vertices $S_k$ belong to the same strongly connected component of $\mathcal{R}_S(G)$.

To see that these are the only vertices in this strongly connected component suppose that there is no path from $v_i$ to $v_j$ in $G$. That is, if $v_i,v_j\in S$ then $v_i\in S_a$, $v_j\in S_b$ where $a\neq b$. By similar reasoning then both vertices must be in different strongly connected components of $\mathcal{R}_S(G)$ as there can be no concatenation of branches $\beta_1*\dots *\beta_n$ from $v_i$ to $v_j$. Hence, the strongly connected component in $\mathcal{R}_{S}(G)$ containing $S_k$ contains only the vertices in $S_k$ implying $\mathcal{R}_{S_k}(\mathbb{S}_k)$ is a strongly connected component of $\mathcal{R}_{S}(G)$.

If $\mathbb{S}_k$ consists of a single vertex $v_i\notin S$ then $S_k$ is empty implying $\mathcal{R}_{S_k}(\mathbb{S}_k)$ is as well. On the otherhand if $v_i\in S$ then, from the above, $\mathcal{R}_{S_k}(\mathbb{S}_k)=\mathbb{S}_k$. The result then follows.
\end{proof}

Lemma \ref{lemma6} implies that the adjacency matrix of $\mathcal{R}_S(G)$ can be written in the block diagonal form 
$$M(\mathcal{R}_S(G))=\left[ \begin{array}{cccc}
M_{S_1}(\mathbb{S}_1)      & 0     & \dots  & 0\\
*        & M_{S_2}(\mathbb{S}_2)   &        & \vdots\\
\vdots   &       & \ddots & 0\\
*        & \dots & *      & M_{S_m}(\mathbb{S}_m) 
\end{array}
\right]$$ where $M_{S_i}(\mathbb{S}_i)$ is the adjacency matrix of $\mathcal{R}_{S_i}(\mathbb{S}_i)$ if $S_i\neq\emptyset$. That is, reductions respect the strongly connected component structure of the graph in the sense that the reduction of a strongly connected component is a again a strongly connected component in the reduced graph. 

Another characteristic of strongly connected components mentioned in section 5 is that all edges off of these components $G\in\mathbb{G}$ can be removed without effecting $\sigma(G)$. That is, if $G=(V,E,\omega)$ and $G^{scc}=(V,E^{scc},\omega)$ where $E^{scc}$ is the set of edges in $E$ belonging to some strongly connected component of $G$ then $\sigma(G)=\sigma(G^{scc})$. 

This suggests an alternate approach to finding a reduction of the graph $G$ by first removing the edges of $G$ which do not belong to any strongly connected component then reducing as usual. 

\begin{proposition}
Let $G\in\mathbb{G}$ and $S\in st(G)$. Then $\mathcal{R}_S(G^{scc})=\mathcal{R}_S(G)^{scc}$ and both $\sigma(\mathcal{R}_S(G^{scc}))$ and $\sigma(\mathcal{R}_S(G)^{scc})$ differ from $\sigma(G)$ by at most $\mathcal{N}(G;S)$.
\end{proposition}

\begin{proof}
Let $G=(V,E,\omega)$ and $S\in st(G)$. As $G$ and $G^{scc}$ have the same strongly connected components then the fact that $S\in st(G^{scc})$ implies via lemma \ref{lemma6} that $\mathcal{R}_S(G)$ and $\mathcal{R}_S(G^{scc})$ have the same strongly connected components as well. Hence, the same is true of $\mathcal{R}_S(G^{scc})$ and $\mathcal{R}_S(G)^{scc}$.

Note that $\mathcal{R}_S(G^{scc})$ has the property that each of its edges is contained in some strongly connected component of this graph. This follows from the fact that every branch in $\mathcal{B}_S(G^{scc})$ is contained in some strongly connected component of $G^{scc}$ so lemma \ref{lemma6} implies there are no edges between the associated components in $\mathcal{R}_S(G^{scc})$. As there are no edges between strongly connected components in $\mathcal{R}_S(G^{scc})$ and $\mathcal{R}_S(G)^{scc}$ then these graphs are identical. 

To compare spectra note that both $\sigma(G)=\sigma(G^{scc})$ and $\sigma(\mathcal{R}_S(G))=$\\ $\sigma(\mathcal{R}_S(G))^{scc}$. Furthermore, an application of theorem \ref{theorem1} implies $\sigma(G^{scc})$ and $\sigma(\mathcal{R}_S(G^{scc}))$ differ at most by $\mathcal{N}(G;S)$ since every loop of $E$ is contained in $E^{scc}$. Similarly, as $\sigma(G)$ and $\sigma(\mathcal{R}_S(G))$ differ at most by $\mathcal{N}(G;S)$ the proof is complete.
\end{proof}

\section{Unique Reductions}

A typical graph will usually have many different structural sets which on the one hand adds flexibility to the process of graph reductions but on the other means some choice of structural set must be made before a graph can be reduced. In this section we consider a particular type of structural set that is defined independent of any specific graph. The advantage in this is foremost that every graph in $\mathbb{G}$ will have a reduction with respect to this particular type of vertex set. This will allow us not only to relate each graph's spectrum to this particular reduction but also to partition the graphs in $\mathbb{G}$ according to their reductions.

For a vertex $v$ let the \textit{out-degree} of this vertex be denoted by $d_{out}(v)$. For the graph $G=(V,E)$ define $$D_{out}(G)=\{v\in V:d_{out}(v)\geq 2\}.$$ 

\begin{definition}
We say a cycle of $G$ is a \textit{simple cycle} if it contains no vertices in $D_{out}(G)$. We call the set of vertices $D_{out}(G)$ together with the vertices on the simple cycles of $G$ the \textit{basic structural set} of $G$ and denote this by $bas(G).$
\end{definition}

To justify our use of the term \textit{basic structural set} we give the following proposition.

\begin{proposition}\label{prop4}
For $G\in\mathbb{G}$ the set $bas(G)\in st(G)$ where $\mathcal{N}(G,bas(G))=\{0\}$. Also the relation of having a reduction in common with respect to a graphs basic structural sets is an equivalence relation on $\mathbb{G}$.
\end{proposition}

\begin{proof}
For $G\in\mathbb{G}$ it follows that $bas(G)\in st(G)$ as every cycle of $G$ contains some vertex in $bas(G)$. This can be seen from the fact that every cycle of $G$ is or is not simple which in either case implies that the cycle must contain a vertex in $bas(G)$. Also, because no branch in $\mathcal{B}_{bas(G)}(G)$ contains a nonzero loop on an interior vertex then $\mathcal{N}(G,bas(G))=\{0\}$.

Next, if $G,H,K\in\mathbb{G}$ such that the pair $G$ and $H$ and the pair $H$ and $K$ have a reduction in common with respect to their basic structural sets then $\mathcal{R}_{bas(G)}(G)\simeq\mathcal{R}_{bas(H)}(H)\simeq\mathcal{R}_{bas(K)}(K)$ implying in particular that $\mathcal{R}_{bas(G)}(G)\simeq\mathcal{R}_{bas(K)}(K)$. Hence, having a common reductions with respect to a graphs regular vertices is a transitive relation on $\mathbb{G}$. From this one can also see that the relation is also reflexive and symmetric. 
\end{proof} 

We denote by $[G]$ the equivalence class of graphs in $\mathbb{G}$ having a common reduction with respect to their basic structural sets which contains the graph $G$. This will be of use in section 7.1.

It is also important to note that other criteria induce an equivalence relation on $\mathbb{G}$ or $\mathbb{G}_\pi$ all that is required is some rule that results in a unique reduction or sequence of reductions of any graph in $\mathbb{G}$ or $\mathbb{G}_\pi$ respectively. This can be summarized as follows.

\begin{theorem}\textbf{(Uniqueness and Equivalence Relations)}
Suppose for any graph $G=(V,E,\omega)$ in $\mathbb{G}_\pi$ that $\tau$ is a rule that selects a unique nonempty subset $\tau(V)\subseteq V$. Then $\tau$ induces an equivalence relation $\sim$ on the set $\mathbb{G}_\pi$ where $G\sim H$ if $\mathcal{R}_{\tau(V)}[G]\simeq\mathcal{R}_{\tau(W)}[H]$ where $V,$ $W$ are the vertex sets of $G$ and $H$ respectively.
\end{theorem}

As an example consider the following rule. For a graph $G=(V,E,\omega)$ where $G\in\mathbb{G}_\pi$ let $m(V)\subseteq V$ be the set of vertices of minimal out degree. If $m(V)\neq V$ then by theorem \ref{theorem-1} $\mathcal{R}_{V\setminus m(V)}[G]$ is uniquely defined and this process may be repeated until all vertices of the resulting graph have the same out degree. As the vertex set $\tau(V)$ of the graph resulting from this sequence of reductions is unique then the relation of having an isomorphic reduction via this rule induces an equivalence relation on $\mathbb{G}_\pi$.

\subsection{Reductions Over Weight Sets}

Unlike the weight set $\mathbb{W}$ used in this paper, it is more typical to consider weighted digraphs $G$ having weights in some subset of $\mathbb{C}$. The tradeoff then for considering $\mathcal{R}_S(G)$ is that although the graph structure is simpler the weights become rational functions. In some sense the weights begin to take on the shape of the characteristic polynomial of $M(G)$. On the other hand, if we wish to reduce the size of the graph, i.e. number of vertices, while maintaining its spectrum along with a particular set of edge weights the following is possible. 

\begin{theorem}\label{theorem5}
Let $\mathbb{U}\subseteq\mathbb{W}$ be a unital subring and suppose $G\in\mathbb{G}$ has weights in $\mathbb{U}$. If $bas(G)=\{v_1,\dots,v_m\}$ and $\ell_i$ is the length of the longest branch in $\mathcal{B}_{ji}(G,bas(G))$ for all $1\leq j\leq m$ then there exists a graph $\mathcal{G}$ with the following properties:\\
(1) $\mathcal{G}\in[G]$ implying $\sigma(G)$ and $\sigma(\mathcal{G})$ differ at most by $\{0\}$.\\ 
(2) $\mathcal{G}$ has weights in $\mathbb{U}$.\\
(3) $\mathcal{G}$ has $m+\sum_{i=1}^m(\ell_i-1)$ vertices.
\end{theorem}

\begin{proof}
We prove the theorem by explicitly constructing the graph $\mathcal{G}$ with the required properties. 

To do so, let $G=(V,E,\omega)$ where $G$ has weights in $\mathbb{U}$ some unital subring of $\mathbb{W}$. For $bas(G)=\{v_1,\dots,v_m\}$ let $$\mathcal{B}_i(G)=\bigcup_{1\leq j\leq m}\mathcal{B}_{ji}(G;bas(G))$$ and let $\gamma^i\in\mathcal{B}_i(G)$ be a branch of maximal length for all $1\leq i\leq m$. Furthermore, let $\Gamma=\{\gamma^1,\dots,\gamma^m\}$ and note each $\gamma^i$ is given by the sequence of vertices $\gamma^i(0),\gamma^i(1),\dots,\gamma^i({\ell_i})$ where $\ell_i$ is its length. We construct $\mathcal{G}$ via the following.

First, let $\mathcal{G}=(\mathcal{V},\mathcal{E},\mu)$ where $\{v_1,\dots,v_m\}=bas(\mathcal{G})\subseteq\mathcal{V}$ and each $\gamma^i\in\mathcal{B}_{bas(\mathcal{G})}(\mathcal{G})$ for $i\leq m$. Moreover, each $\gamma^i$ is independent of each $\gamma^j$, $i\neq j$.

If $e_{j,j+1}^i$ is the edge from $\gamma^i(j)$ to $\gamma^i(j+1)$ then set 
$$\mu(e_{j-1,j}^i)=\begin{cases}
                    \prod_{k=1}^{n_i}\omega(e_{k-1,k}^i) \ \text{for} \ j=1\\
                   1, \ \text{otherwise}
                   \end{cases}$$
and note that each $\mu(e_{j-1,j}^i)\in \mathbb{U}$. 

Lastly, for each $\beta\in\mathcal{B}_{ij}(G;bas(G))\setminus \Gamma$, if $\beta$ has length $n$ then add an edge $\beta_e$ from $v_i$ to $\gamma^j(n_j+1-n)$ to the edge set $\mathcal{E}$. Make $\mu(\beta_e)=\prod_{k=1}^{l}\omega_k$ where $\Omega(\beta)=\omega_1,0,\omega_2\dots,\omega_{l-1},0,\omega_l$ is the weight sequence of $\beta$. If multiple edges of $\mathcal{G}$ are reduced to single edges by summing the corresponding weights then $\mathcal{G}$ has the following properties.

The vertex set $\mathcal{E}$ consists of all distinct vertices in $\Gamma$ of which there are $m+\sum_{i=1}^m(\ell_i-1)$. Moreover, each weight of $\mathcal{G}$ is the product or sum of products of elements of $\mathbb{U}$ implying $\mathcal{G}$ has weights in this set. Finally, there is a one-to-one correspondence between branches in $\mathcal{B}_{ij}(G,bas(G))$ and $\mathcal{B}_{ij}(\mathcal{G};bas(\mathcal{G}))$ where corresponding branches have the same branch product. Hence, $\mathcal{G}\in [G]$.
\end{proof}

\begin{figure}
  \begin{center}
    \begin{overpic}[scale=.5]{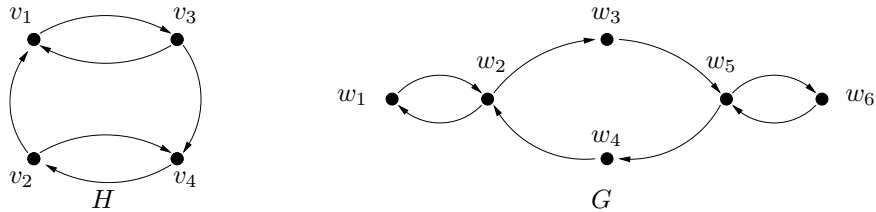}
    \put(10,-3){$H$}
    \put(0,0){$v_2$}
    \put(0,20){$v_1$}
    \put(20,20){$v_3$}
    \put(20,0){$v_4$}
    \put(71,-3){$G$}
    \put(40,10){$w_1$}
    \put(57,14){$w_2$}
    \put(71,20){$w_3$}
    \put(71,5){$w_4$}
    \put(85,14){$w_5$}
    \put(102,10){$w_6$}
    \end{overpic}
  \end{center}
  \caption{$H$ is a reduction of $G$ over the weight set $\{1\}$.}
\end{figure} 

An example of this construction is the graph $H$ in figure 8 which is constructed from the graph $G$ (in the same figure) over the weight set $\mathbb{U}=\{1\}$. $H$ is a reduction over the weight set $\mathbb{U}=\{1\}$ of $G$ in the sense that it has fewer vertices than the graph $G$ from which it is constructed. Furthermore, one can compute $\sigma(G)=\{\pm\sqrt{2},0,0,0,0\}$ and $\sigma(H)=\{\pm\sqrt{2},0,0\}$. 

A more complicated problem involves finding the graph $H\in[G]$ with the least number of vertices where both $H$ and $G$ have weights in some set $\mathbb{U}\subseteq\mathbb{W}$.

\section{Concluding Remarks}
The main results of this paper are concerned with the way in which the structure of a graph influences the spectrum of the graph's adjacency matrix and to what extent this spectrum can be maintained if this structure is simplified. For the most part these results give algorithmic methods whereby a graph can be reduced in size but do not mention if such methods might be useful in determining the spectrum of a given graph. 

As is shown in \cite{BW09}, the graph reductions considered here do indeed help in the estimation of a graph's spectrum. For example, it is possible to extend the eigenvalue estimates given by the classical result of Gershgorin \cite{Gershgorin31} to matrices with entries in $\mathbb{W}$. The main results of \cite{BW09} is that eigenvalue estimates via this extension improve as the graph is reduced. Analogous results and extensions also hold for the work done by Brauer, Brualdi and Varga \cite{Brauer47,Brualdi82,Hadamard03,Horn85,Varga09}whose original results are each improvements of Gershgorin's. Gershgorin's original result is in fact equivalent to a nonsingularity result for diagonally dominant matrices (see theorem 1.4 \cite{Varga09}) which can be traced back to earlier work done by L\'{e}vy, Desplanques, Minkowski, and Hadamard \cite{Levy81,Desplanques87,Minkowski00,Hadamard03}. 

Importantly, we note that via our method of graph reductions we obtain better estimates than those given by all of the previous existing methods. Moreover, graph reductions can be used to obtain estimates of the spectrum of a matrix with increasing precision depending on how much one is willing to reduce the associated graph. If the graph is completely reduced the corresponding eigenvalue estimates give the exact spectum of the matrix along with some finite set of points. 

These techniques can furthermore be used for the estimation of spectra for combinatorial and normalized Laplacian matrices as well as giving bounds on the spectral radius of a given matrix. In fact it is in such applications that the flexibility of isospectral graph reductions is particularly useful.

The results of the present paper demonstrate various approaches to simplifying a graph's structure while maintaining its spectrum. Therefore, these techniques can be used for optimal design, in the sense of structure simplicity of dynamical networks with prescribed dynamical properties ranging from synchronizability to chaoticity \cite{Afriamovich07, Blank06}. 

\section{Acknowledgments}
We would like to thank C. Kirst and M. Timme for valuable comments and discussions. This work was partially supported by the NSF grant DMS-0900945 and the Humboldt Foundation.

\end{document}